\documentclass[a4paper,12pt]{amsart}
\usepackage{amsmath,amssymb}
\usepackage[shortlabels]{enumitem}
\usepackage{xcolor}

\def\d{\delta}

\def\D{\mathbb D}
\def\C{\mathbb C}
\def\R{\mathbb R}
\def\N{\mathbb N}

\def\bcases{\begin{cases}}
\def\ecases{\end{cases}}

\def\wdt{\widetilde}
\def\ds{\displaystyle}

\newcommand{\cl} {\overline}

\newtheorem{thm}{Theorem}
\newtheorem{prop}[thm]{Proposition}
\newtheorem{cor}[thm]{Corollary}
\newtheorem{lem}[thm]{Lemma}

\renewcommand{\Im}{\operatorname{\rm{Im}}}
\renewcommand{\Re}{\operatorname{\rm{Re}}}
\newcommand{\eps}{\varepsilon}

\newenvironment{proof*}{\vskip 2mm\noindent {}}{\hfill $\Box$ \vskip 2mm}

\title{Invariant metrics of model domains near pseudoconcave points}

\author{Nikolai Nikolov}
\address{N. Nikolov\\
Institute of Mathematics and Informatics\\
Bulgarian Academy of Sciences\\
Acad. G. Bonchev 8, 1113 Sofia, Bulgaria
\vspace{1mm}
\newline Faculty of Information Sciences\\
State University of Library Studies and Information Technologies\\
Shipchenski prohod 69A, 1574 Sofia,
Bulgaria}

\email{nik@math.bas.bg}

\author{Pascal J. Thomas}
\address{Thomas, Pascal J.\\
Univ Toulouse, INUC, UT2J, INSA Toulouse, TSE, CNRS, IMT, 
Toulouse, France}

\email{pascal.thomas@math.univ-toulouse.fr}

\thanks{The first named author was partially supported by the Bulgarian National
Science Fund, Ministry of Education and Science of Bulgaria under contract KP-06-N82/6. The
second named author wishes to thank the Institute of Mathematics and Informatics of the Bulgarian Academy
of Sciences for its hospitality during the time when most of this work was carried out.}

\subjclass[2010]{32F45}

\begin{document}

\keywords{Kobayashi-Royden, Kobayashi-Buseman \& Sibony metrics.}

\begin{abstract}
We give precise estimates of some holomorphically invariant infinitesimal metrics near a
pseudoconcave point in a wide family of ``model'' domains for that situation in $\C^2$. This extends
to metrics (rather than distances)
the authors' previous results from \cite[Section 4]{NT}, and also takes into account
defining functions more general than just power functions.
\end{abstract}

\maketitle

\section{Introduction and results}

\subsection{Motivations.}

Holomorphically invariant metrics and distances have long proved useful for various questions in complex analysis and geometry, notably
the mapping problem between domains of $\C^d$. It turns out that the presence of non-pseudoconvex points on the boundary
of the domain gives rise to quantitatively different behavior of the invariant metrics in a neighborhood of those points;
this was explored in \cite{Kr, Fu1, For-Lee, DNT} among others. The interested reader should also consult the survey \cite{Fu2}.

Typically, failure of pseudoconvexity at a boundary point $p$ of a domain $\Omega$ is exploited by showing that some pseudoconcave model domain
can be found with $G\subset \Omega$ and $p\in (\partial G) \cap (\partial \Omega)$.  The behavior of holomorphic invariants
in a range of such domains was studied in \cite[Section 4]{NT}. Here we want to concentrate on infinitesimal metrics rather than distances, and generalize this to a wider range of models.  Note however that we could have considered more general models (with a defining function
also depending on $\Im z_1$ for instance, see \cite{NT}) and extended those estimates to higher dimensions.

\subsection{Basic definitions.}

Let $\D$ stand for the unit disc in the complex plane.
For an open set $\Omega \subset \C^d$, $\mathcal O(\D,\Omega)$ is the set of holomorphic maps from $\D$ to $\Omega$ and
the Kobayashi-Royden metric is defined as follows:
$$
\kappa_\Omega(z;X):=\inf\{\lambda^{-1} :\lambda >0, \exists\varphi\in\mathcal O(\D,\Omega), \varphi(0)=z, \varphi'(0)=\lambda X\}.
$$
This is a Finsler metric (that is, $\kappa_\Omega(z; \alpha X) = |\alpha |\kappa_\Omega(z;X)$, $\alpha \in \C$),
contracting under holomorphic maps, and therefore invariant under biholomorphisms.
It does not always satisfy the triangle
inequality, so it makes sense to define,
for $m\in \N$, $m\ge 1$, the $m$-th Kobayashi-Royden metrics
\[
\kappa^{(m)}_\Omega(z;X):=\inf\{\sum_{j=1}^m \kappa_\Omega(z;X_j) : X_j\in \C^d, 1\le j \le m, \mbox{ and }\sum_{j=1}^m X_j=X\}.
\]
Clearly, $\kappa^{(m)}_\Omega \ge \kappa^{(m+1)}_\Omega$;
$\hat \kappa_\Omega:=\lim_{m\to\infty} \kappa^{(m)}_\Omega$ is the Kobayashi-Buseman metric, the largest Finsler metric
less than the Kobayashi-Royden metric which satisfies the triangle
inequality.

As in \cite{DNT}, we define a metric that sits between the Kobayashi-Royden and Kobayashi-Buseman metrics.
Recall that the
{\it indicatrix} of a metric $M_D$ at a base point $z$ is
$$
I_z M_D:= \left\lbrace v \in T^{\C}_z D: M_D(z,v)<1 \right\rbrace .
$$
The indicatrix of a metric which satisfies the triangle inequality is convex.
The larger the indicatrix, the smaller the metric. We  define $\wdt \kappa_D$ to be {\it the largest
invariant metric with pseudoconvex indicatrices,} i.e., since $\kappa_D$ is the largest invariant metric,
 $I_z\wdt \kappa_D$ is the envelope of holomorphy of
$I_z \kappa_D$ for any $z\in D.$

Using plurisubharmonic functions, we can define
 the Sibony metric of a domain $D$: for $p\in D$, $X\in \C^d$,
\[
S_D(p,X):=\sup\left\{ \partial\bar\partial u(p)(X, \bar X)^{1/2}
:=\left(\sum_{i,j=1}^d\frac{\partial^2 u}{\partial z_i\partial \bar z_j}(p) X_i\bar X_j\right)^{1/2}: u\in A(p,D)\right\},
\]
where $\D$ denotes a unit disc in $\C$, and $A(p,D)$ is the set of (plurisubharmonic) functions on $D$ such that
$u(p)=0$, $u$ is $\mathcal C^2$ near $p$, $\log u$ is plurisubharmonic on $D$, and $0\le u\le 1$ on $D$.
The Sibony metric satisfies the triangle inequality, see for instance \cite[Lemma 2]{For-Lee}.

Recall that for a domain $\Omega\subset \C^d$,
\begin{equation}
\label{order}
S_\Omega\le\hat \kappa_\Omega\le\wdt \kappa_\Omega\le \kappa_\Omega.
\end{equation}


\subsection{Results.}

For $\psi$ a continuous function from $[0,1]$ to $[0,\infty)$ with $\psi(0)=0$, $\psi(1)>0$ let us define
the domain
\begin{equation}
\label{defG}
G_\psi := \{z\in \D^2: \Re z_1 < \psi(|z_2|) \}.
\end{equation}
Note that this is always connected because $\psi(x)\ge 0$, and open because $\psi$ is continuous.  
\smallskip

Let $(x_N,x_T)\in\C^2$ and for $\delta\in(0;1)$,  $p_\delta:=(-\delta,0) \in G_\psi$.
\smallskip

{\bf Standing assumption.}
We will restrict consideration to strictly increasing functions $\psi$, and
$\delta \in (0;\delta_0]$, where $\delta_0 <1 $ and is chosen depending on $\psi$. 
In particular, we choose $\delta_0 < \sup_{[0,1)}\psi$, so that $\psi^{-1}(\delta)$ always make sense.
In all the estimates below, the constants which appear depend on $\psi$ and $\delta_0$.

\begin{thm}
\label{othermet}
If $\frac{\psi(x)}{x}$ is an increasing function on $(0;1)$, there exist $0<c<C$ such that
\[
c \left(\frac{\psi^{-1}(\delta)}{\delta} |x_N| + |x_T|\right)  \le
S_{G_\psi}(p_\delta;(x_N,x_T)) \le
M_{G_\psi}(p_\delta;(x_N,x_T)) \le C \left(\frac{\psi^{-1}(\delta)}{\delta} |x_N| + |x_T|\right),
\]
where $M_{G_\psi}$ stands for either $\kappa^{(2)}_{G_\psi}$ or $\wdt \kappa_{G_\psi}$.
\end{thm}

\noindent
{\bf Remarks 1.}

(a) When $\psi(x)=x^2$, the upper bounds for $\wdt \kappa_{G_\psi}$ were already obtained in \cite[Lemma 7]{DNT},
as well as the lower bound for $S_{G_\psi}$ when $\psi(x)=x^{1+\eps}$, $0\le \eps \le 1$ \cite[Corollary 6]{DNT}.

(b) For $M$ any holomorphically contracting metric (with $M_\D(0;1)=1$),
using the fact that $\D \ni \zeta \mapsto (-\delta, \zeta) \in G_\psi$, $M_{G_\psi} (p_\delta ; (0,1))\le 1$.
On the other hand, it follows by inclusion of
domains that $M_{G_{\psi}}\ge M_{\D^2}$, therefore $M_{G_\psi}(p_\delta;(x_N,x_T))\ge\max\{|x_N|, |x_T|\}$.
In particular, $M_{G_\psi} (p_\delta ; (0,1))=1$.

(c) If $\psi_1 \le \psi_2$, $M_{G_{\psi_1}}\ge M_{G_{\psi_2}}$. From this and
Theorem \ref{othermet} applied to $\psi_2(x)=  x$, we deduce that if $\psi_1(x) \gtrsim x$, and $X \in \C^2$,
then
$S_{G_{\psi_1}}(p_\delta;X)$, $\kappa^{(2)}_{G_{\psi_1}}(p_\delta;X)$,
and $ \wdt \kappa_{G_{\psi_1}}(p_\delta;X)$ are all comparable to $|X|$.

(d) The Carath\'eodory-Reiffen metric is the smallest of the holomorphically contracting metrics
with $M_\D(0;1)=1$, so always provides a lower bound to the above quantities.
Recall that, for $p\in \Omega \subset \C^n$ and $X\in \C^n$, the Carath\'eodory-Reiffen metric is given by
\[
\gamma_\Omega (p; X) := \sup\left\{ |Df(p) \cdot X|: f\in \mathcal O (\Omega, \D), f(p)=0\right\}.
\]
If $\psi$ is increasing, by the Hartogs phenomenon, any $f\in \mathcal O (G_\psi, \D)$ extends to a holomorphic function on
$D(0, \sup_{[0;1)}\psi)\times \D$, so $\gamma_\Omega (p_\delta; (x_N,x_T)) \le c( |x_N| + |x_T|)$.  This was already
pointed out in \cite{For-Lee} for instance.
\smallskip

The next proposition sums up some simple facts about the Kobayashi-Royden
metric taken on directions close to the ``tangential'' direction $(0,1)$.

\begin{prop}\label{tangent}

\

\begin{enumerate}
\item
If $\psi(x)\ge c_0 x$ for some $c_0>0$ and $|x_N|\le \min(1,c_0) |x_T|$, then
\[
|x_T|\le \kappa_{G_\psi} (p_\delta; (x_N,x_T))\le
\max \left(1,  \frac{|x_N|}{|x_T|}\cdot\frac{1}{1-\delta}\right) |x_T|.
\]
\item
If $\frac{\psi(x)}{x}$ is a strictly increasing function on $(0;1)$, 
if $0<\delta \le \delta^*$, where $\delta^*$ is the unique solution
of $1-\delta= \frac{\delta}{\psi^{-1}(\delta)}$,
and if 
$|x_N|\le \min \left( 1, \frac{\delta}{\psi^{-1}(\delta)}\right) |x_T|$,
then
$\kappa_{G_\psi}(p_\delta;(x_N,x_T)) = |x_T|$.
\end{enumerate}
\end{prop}

Notice that in Part (1), when $|x_N|<|x_T|$ and
$\delta$ is small enough,
$\kappa_{G_\psi} (p_\delta; (x_N,x_T))=|x_T|$.

Part (2) was obtained for the case $\psi(x)=x^2$ in \cite[Lemma 7(1)]{DNT}.

The estimates
for vectors close to the normal direction $(1,0)$ are a considerably more delicate
matter; \cite[Proposition 3]{DNT} shows that the lower estimates cannot be the same as the upper estimates.
See also, for slightly different model domains,
the results of Fu \cite[Proposition 2.3]{Fu2}.

Typical examples for our various situations are provided by the power functions
$\psi (x) = x^\beta$, $\beta>0$. We begin by comparing $\psi(x)$ to $x^{1/2}$.

\begin{prop}
\label{onehalf}

\

\begin{enumerate}
\item
If  $\psi(x)\ge c_0 \sqrt x$  for some $c_0>0$,
then  there exists $C\ge1$  such that
\[
\max(|x_N|,|x_T|) \le \kappa_{G_\psi}(p_\delta;(x_N,x_T))
\le C \max(|x_N|,|x_T|).
\]

\item
If $\frac{\psi(x)}{\sqrt x}$ is an increasing function on $(0;1)$,
there exists $c>0$
such that
\[
c F_2 \left(\delta, \left|\frac{x_T}{x_N}\right|\right)
|x_N| \le \kappa_{G_\psi}(p_\delta;(x_N,x_T)),
\]
where
\[
F_2(\delta,t):= \min \left(\frac{\sqrt{\psi^{-1}(\delta)}}\delta ,
\frac{t}{ \psi(t^2)}
 \right).
\]
\end{enumerate}

\end{prop}

\noindent
{\bf Remarks 2.}

(a) The paradigm for Part (1) is given by $\psi(x)= \sqrt x$
and $G_\psi$ then looks like a bidisc with a thin cuspidal wedge removed.
The result suggests that such inner cusps are ``too thin
to be seen'' by the Kobayashi-Royden metric at
points within the axis of the cusp.  On the other hand, when $\frac{\psi(x)}{\sqrt x}$ is
increasing, then $\psi(x) \le \psi(1) \sqrt x$.

(b) Part (2) generalizes to vectors close to the normal direction
the lower bound that had been obtained in \cite[Lemma 8]{DNT},
in the case $c_0 \le \left|\frac{x_T}{x_N} \right| \le  \delta^{\frac1\beta - 1}$,
for the  domain
\[
\Omega_\beta :=\left\{ \Re z_1 < |z_2|^\beta + |\Im z_1|^\beta \right\}, \quad \beta >1.
\]

(c) Theorem \ref{main} (1) gives a sharper result than Proposition \ref{onehalf} (2), under
an additional hypothesis about the growth of $\psi$.

The next theorem provides estimates mostly for the case
where the vector $X$ remains in a cone around the normal
direction $(1,0)$.  Again, we make assumptions about the growth
of $\psi$ compared to that of $\sqrt x$ and of $x$.

\begin{thm}\label{main}

\

\begin{enumerate}
\item
Let $\frac{\psi(x)}{\sqrt x}$ be an increasing function
and let $x\mapsto \psi_1(x):=\frac{\psi(x)}{x}$ be a decreasing function on $(0;1)$,
with a ``halving'' property:

\noindent(H) there exists $K>1$ such that for any $x\in (0;1/K]$, $\psi_1(Kx)\le \frac12 \psi_1(x)$.

\noindent Then there exist $0<c<C$ such that, if $|x_T|\le|x_N|$,
\[
c F_3 \left(\delta,\left|\frac{x_T}{x_N}\right|\right)  |x_N| \le \kappa_{G_\psi}(p_\delta;(x_N,x_T))
\le
C F_3 \left(\delta,\left|\frac{x_T}{x_N}\right|\right)  |x_N|,
\]
where
\[
F_3(\delta,t):= \min \left(\frac{\sqrt{\psi^{-1}(\delta)}}{\delta} ,
\frac{8 t} {\sqrt{\psi_1^{-1}\left(\frac{1}{8t}\right)} } \right).
\]

\item
Let $\psi(x)=x$. Then there exist $0<c<C$ such that for any $(x_N,x_T)\in \C^2$,
\[
c \max\left( \frac{|x_N|-|x_T|}{\sqrt \delta}, |x_T| \right)
\le 
\kappa_{G_\psi}(p_\delta;(x_N,x_T))
\le
C \max\left( \frac{|x_N|-|x_T|}{\sqrt \delta}, |x_T| \right).
\]

\item
If $\frac{\psi(x)}{x}$ is an increasing function on $(0;1)$, then
there exist $C>0$ such that if $0<\delta<\psi(1)$, and
if $|x_T| \le\frac{\psi^{-1}(\delta)}{8\delta}|x_N|$,
\[
\frac1{4\sqrt 2} \frac{\sqrt{\psi^{-1}(\delta)}}{\delta}|x_N|
\le
\kappa_{G_\psi}(p_\delta;(x_N,x_T)) \le
C  \frac{\sqrt{\psi^{-1}(\delta)}}{\delta}|x_N|.
\]
\end{enumerate}
\end{thm}

Typical examples are provided by the power functions
$\psi (x) = x^\beta$.
For the reader's convenience, we collect the consequences of Propositions \ref{tangent} and Theorem \ref{main}
for the power functions in the corollary below; the case $\beta=1$
is already dealt with in Theorem \ref{main} (2).

\begin{cor}
\label{power}
Let $\psi (x) = x^\beta$.  Then there exist $0<c<C$, depending on $\beta$, such that:
\begin{enumerate}
\item
If $0<\beta \le \frac12$ (cf. Proposition \ref{onehalf} (1)), 
\newline
$\max(|x_N|,|x_T|) \le \kappa_{G_\psi}(p_\delta;(x_N,x_T))\le C \max(|x_N|,|x_T|)$.
\item
If $\frac12 < \beta <1$ (cf. Theorem \ref{main} (1) Proposition \ref{onehalf} (1), Proposition \ref{tangent} (1)), then:
\begin{itemize}
\item
if $0\le |x_T| \le \delta^{\frac1\beta-1} |x_N|$, then $c \delta^{\frac1{2\beta}-1} |x_N| 
\le \kappa_{G_\psi}(p_\delta;(x_N,x_T)) \le C \delta^{\frac1{2\beta}-1} |x_N|$;
\item
if $ \delta^{\frac1\beta-1} |x_N| \le |x_T| \le |x_N| $, then 
\\
$c \left|\frac{x_T}{x_N}\right|^{1-\frac1{2(1-\beta)}} |x_N|  \le \kappa_{G_\psi}(p_\delta;(x_N,x_T)) \le C \left|\frac{x_T}{x_N}\right|^{1-\frac1{2(1-\beta)}} |x_N| $;
\item
if $ |x_N| \le |x_T| \le \frac1{1-\delta} |x_N| $, then $|x_T| \le \kappa_{G_\psi}(p_\delta;(x_N,x_T)) \le  \frac1{1-\delta} |x_N|$;
\item
if $\frac1{1-\delta}  |x_N| \le |x_T| $, then $ \kappa_{G_\psi}(p_\delta;(x_N,x_T))=|x_T| $.
\end{itemize}
\item
If $1< \beta $ (cf. Theorem \ref{main} (3) and Proposition \ref{tangent} (2)), then:
\begin{itemize}
\item
if $|x_T| \le\frac18 \delta^{\frac1\beta-1}|x_N|$, then 
\\
$\frac1{4\sqrt 2} \delta^{\frac1{2\beta}-1}|x_N|
\le
\kappa_{G_\psi}(p_\delta;(x_N,x_T)) \le
C  \delta^{\frac1{2\beta}-1}|x_N|$;
\item
if $ \delta^{\frac1\beta-1}|x_N| \le |x_T| $, and  if $0<\delta \le \delta^*$, where $\delta^*$ is the unique solution
of $1-\delta= \delta^{1-\frac1\beta}$,
then
$ \kappa_{G_\psi}(p_\delta;(x_N,x_T))=|x_T| $.
\end{itemize}
\end{enumerate}
\end{cor}

A more precise result than ours for the case $\beta=2$ was obtained in \cite[Lemma 7, (1)]{DNT},
but Proposition \ref{onehalf}  (2) and
Theorem \ref{main} (3) generalize \cite[Lemma 7, (2), (3)]{DNT}.
\smallskip

\noindent
{\bf Remarks 3.}

(a) In Theorem \ref{main} (2), one could have considered $\psi(x)=c_0 x$.  However, we can reduce ourselves to
$c_0=1$ by a linear change of coordinates and using a localization lemma for the Kobayashi-Royden metric, 
at the cost of changing the multiplicative constants.  We omit the details. 

(b) There is a gap between the ranges
$\ds  8 \le  \frac{\psi^{-1}(\delta)}{\delta} \frac{|x_N|}{|x_T|}$
from Theorem \ref{main} (3) and
$\ds \frac{\psi^{-1}(\delta)}{\delta} \frac{|x_N|}{|x_T|} \le 1$ from Proposition \ref{tangent} (2),
and it is essential, since
the behavior of the Kobayashi-Royden changes radically between those two. This generalizes \cite[Proposition 3 (2)]{DNT}.
We do not know the precise behavior of the metric in the intermediate range.

(c) The reader might be surprised that
in Theorem \ref{main} (1), when
$ |x_T|>\frac{\psi^{-1}(\delta)}{8\delta} |x_N| $,
then the bounds do not depend explicitly
on $\delta$.  For instance, when $|x_N| =|x_T|$, $F_3$ becomes constant. 
This is not so surprising
if we recall that the metric $\kappa_{G_\psi} (p_\delta,X)$
depends on two variables, the point $p_\delta$
and the vector $X$.

One can reduce this to a one-variable situation in the following way: if
vectors $(x_N(\delta), x_T(\delta))$ are given, then
$\kappa_{G_\psi}\left(p_\delta;(x_N(\delta), x_T(\delta))\right)$ will depend
on the behavior of $|x_T(\delta)|/|x_N(\delta)| $
as $\delta\to0$.  If it tends to $0$ faster than
$\frac{\psi^{-1}(\delta)}{\delta}  $, then the metric
will behave like $\frac{\sqrt{\psi^{-1}(\delta)}}{\delta} |x_N|$;
if not, it will be some function depending on how slowly
$|x_T(\delta)|/|x_N(\delta)| $ goes to $0$, if it does,
and bounded below if it does not.

(d) In the situation of Corollary \ref{power}, with $\beta\in(\frac12;1)$,
and a family of vectors given by $X(\delta)=
(1,\delta^\gamma)$, with $\gamma>0$, then
if $\gamma \in (\frac{1}{\beta}-1,\infty)$,
applying part (1)
we have
\[
c \delta^{\frac{1}{2\beta}-1} \le
\kappa_{G_\psi}(p_\delta;X(\delta)) \le C
\delta^{\frac{1}{2\beta}-1}.
\]
Parts (2) and (3) show that the same upper bound
holds for all $\beta\ge \frac12$.

On the other hand, if $\gamma \in (0,\frac{1}{\beta}-1]$,
we see that the rate of blow-up depends on $\gamma$
(notice that the exponent is a linear interpolation between the two extreme cases):
\[
c \delta^{-\gamma \frac{2\beta-1}{2(1-\beta)}}
\le
\kappa_{G_\psi}(p_\delta;X(\delta)) \le
C \delta^{-\gamma \frac{2\beta-1}{2(1-\beta)}}.
\]
Finally, the case of a constant vector with
non-zero $x_T$-component (so $\gamma=0$)
would yield a bounded estimate.
\vskip.3cm

The proof of Theorem \ref{othermet} is given in Section \ref{other},
of Propositions \ref{tangent} and \ref{onehalf} in Section \ref{props},
and of Theorem \ref{main} in Section \ref{mainpf}.

\section{Proof of Theorem \ref{othermet}}
\label{other}

\subsection{Upper estimate for $\kappa^{(2)}_{G_\psi}$}

Consider the map given, for $\zeta\in \D$, by
\begin{equation}
\label{mapk2}
\varphi(\zeta)= \left(-\delta+ \frac{\delta}{\psi^{-1}(\delta)} \zeta, \zeta\right).
\end{equation}
The fact that $\Re(\varphi_1(\zeta)) < \psi(|\varphi_2(\zeta)|) $ is clear when $|\zeta|\le \psi^{-1}(\delta)$; otherwise, 
use $\ds \frac{\delta}{\psi^{-1}(\delta)} \le \frac{\psi(|\zeta|)}{|\zeta|}$, which once again yields the inequality.

Choose $\delta_0 <\min(\frac12, \sup_{[0;1)} \psi)$. 
Then for any $\delta\le\delta_0$, 
\[
|\varphi_1(\zeta)| \le \frac12 + \frac{\delta_0}{\psi^{-1}(\delta_0)} |\zeta|,
\]
so for $|\zeta|<r_0:= \frac{\psi^{-1}(\delta_0)}{2\delta_0}$,
$\varphi(\zeta)\in G_\psi$. 

Rescaling $\zeta$ if it turns out that $r_0<1$,
 $\kappa_{G_\psi}(p_\delta;(1, \frac{\psi^{-1}(\delta)}{\delta})) \le C \frac{\psi^{-1}(\delta)}{\delta} $.

We saw in Remark 1(b) that $\kappa_{G_\psi}(p_\delta;(0,1)) \le 1$.

Since we have a Finsler metric, it is enough to obtain the desired estimate for $x_N=1$, and
\[
\kappa^{(2)}_{G_\psi}(p_\delta;(1,x_T))  \le
\kappa_{G_\psi}(p_\delta;(1, \frac{\psi^{-1}(\delta)}{\delta})) +
\left| x_T - \frac{\psi^{-1}(\delta)}{\delta} \right| \kappa_{G_\psi}(p_\delta;(0,1))
\le C\left( \frac{\psi^{-1}(\delta)}{\delta} + |x_T|\right).
\]

\subsection{Upper estimate for $\wdt \kappa_{G_\psi}$}

The  estimate will be established if we show that
$D(0,  \frac{\delta}{\psi^{-1}(\delta)} ) \times \D \subset I_z \wdt \kappa_{G_\psi}$.
By the Hartogs phenomenon, this will be true if
$D(0,  \frac{\delta}{\psi^{-1}(\delta)} ) \times \partial \D \subset \cl{I_z \kappa_{G_\psi}}$.

For $|a|<\frac{\delta}{\psi^{-1}(\delta)}$, $\theta \in \R$, consider the map given by
\[
\D \ni \zeta\mapsto \varphi(\zeta)= \left(-\delta+ a \zeta, e^{i\theta} \zeta\right).
\]
Since $\Re (\varphi_1(\zeta)) \le \frac{\delta}{\psi^{-1}(\delta)} |\zeta|$
and $|\varphi_2(\zeta)|= |\zeta|$, seeing that $\varphi(\zeta)\in G_\psi$ reduces to
the same computation as for the map $\varphi$ in 
\eqref{mapk2} above. 

\subsection{Lower estimate for $S_{G_\psi}$}

Since the Sibony metric verifies the triangle inequality, it will be enough to estimate it on
each of the basis vectors. The estimate for $(0,1)$ was given in Remark 1(b).

To estimate $S_{G_\psi} (p_\delta ; (1,0))$ from below, we construct a function $u$ that will be a candidate for
the supremum in the definition of the Sibony metric.

Let $C_1:= \psi^{-1}(\delta/2)$, $C_2:=C_1^2$. We define
\[
f(\zeta):= C_2 \left| \frac{\zeta+\delta}{\zeta-\delta}\right|^2,
\mbox{ then }\frac{\partial^2 f}{\partial \zeta  \partial \bar \zeta}(-\delta) = \frac{C_2}{4\delta^2},
\]
and then $u:= e^v$, with
\[
v(z_1,z_2):=
\begin{cases}
\max \left( \log(f(z_1) + |z_2|^2), \log|z_2| + C_3 \right)-C_4, \mbox{ for } |z_2| \le C_1,
\\
\log |z_2| +C_3 - C_4, \mbox{ for } |z_2| \ge C_1 .
\end{cases}
\]
When $|z_2| \le C_1,$ $z\in G_\psi$, then $\Re z_1 \le \psi (C_1) = \delta/2$.  When $\Re \zeta \le \delta/2$,
$\left| \frac{\zeta+\delta}{\zeta-\delta}\right|^2$ is maximal for $\Re \zeta = \delta/2$, and for $y\in \R$,
\[
\left| \frac{\frac32 \delta+iy}{-\frac12 \delta +iy}\right| = \left| \frac{ \delta}{\frac12 \delta -iy}+ \frac{\frac12 \delta+iy}{\frac12 \delta -iy}\right| \le 2 + 1.
\]
Hence for $|z_2|= C_1$, $f(z_1) + |z_2|^2 \le (9 + 1) C_1^2$, and if we take $C_3 > \log(10 C_1)$,
\newline
$\log(f(z_1) + |z_2|^2) < C_3+\log |z_2|$,
and this inequality holds in a neighborhood.
This implies that $v$ is plurisubharmonic.  Finally we choose $C_4\ge C_3+\log |C_1|$, so that $v\le 0$ on $G_\psi$.

For $\Re z_1<0$, $u(z_1,0)= f(z_1)$, so $\frac{\partial^2 u}{\partial z_1 \partial \bar z_1}(p_\delta) = \frac{\psi^{-1}(\delta/2)^2}{4\delta^2}$;
passing to the square root, we see that $S_{G_\psi} (p_\delta ; (1,0)) \ge \frac{\psi^{-1}(\delta/2)}{2\delta}$.

It remains to see that $\psi^{-1}(\delta ) \le C \psi^{-1}(\delta/2)$.  Observe that to have this, it is
enough to assume that $\psi(x)/x^\gamma$ is increasing for some $\gamma >0$.  Indeed, since $\psi^{-1}$ is an increasing
function,
\[
\frac{\psi^{-1}(x)}{x^{1/\gamma}} = \left( \frac{\psi^{-1}(x)^\gamma}{\psi(\psi^{-1}(x))} \right)^{1/\gamma}
\]
must be decreasing. So
\[
\psi^{-1}(\delta/2) = (\delta/2)^{1/\gamma} \frac{\psi^{-1}(\delta/2)}{(\delta/2)^{1/\gamma}} \ge
(\delta/2)^{1/\gamma} \frac{\psi^{-1}(\delta)}{\delta^{1/\gamma}} = 2^{-1/\gamma}\psi^{-1}(\delta).
\]

\section{Proofs of Propositions \ref{tangent} and \ref{onehalf}}
\label{props}

\subsection{Proof of Proposition \ref{tangent}}

Using the fact that
$\kappa_{G_\psi}(p_\delta;(x_N,x_T))=
|x_T|\kappa_{G_\psi}(p_\delta;(\frac{x_N}{x_T},1))$, we may assume $x_T=1$.

The lower estimate in Part (1) follows from Remark 1(b). For the upper
estimate, it is enough to consider $\psi(x)=c_0 x$ by inclusion of domains.

When $\delta \le 1-|x_N|$, $\max \left(1,  |x_N|(1-\delta)^{-1}  \right)= 1$.
Consider the map
\[
\varphi(\zeta):=
\left( -\delta + x_N \zeta , \zeta \right).
\]
The condition on $x_N$ ensures $\varphi(\D)\subset \D^2$.
For $\Re \zeta \le 0$, we  have $\Re\varphi_1(\zeta)<0$ so $\varphi(\zeta)\in G_\psi$,
for $\Re \zeta > 0$, $|\varphi_2(\zeta)|=|\zeta|> \Re\varphi_1(\zeta)$.

When $\delta > 1-|x_N|$, consider
\[
\varphi(\zeta):=
\left( -\delta + \lambda x_N \zeta , \lambda \zeta \right), \mbox{ with } \lambda = \frac{1-\delta}{|x_N|}<1.
\]
Again it is easy to check $\varphi(\D)\subset \D^2$ and for $\Re \zeta >0$,
$|\varphi_2(\zeta)|=\lambda|\zeta|\ge \lambda |  x_N \zeta|>\Re\varphi_1(\zeta)$.

In Part (2), the
lower estimate again holds by Remark 1(b), while the upper estimate
follows from considering the map $\varphi(\zeta):= p_\delta+ \zeta (x_N,1)$:
again the condition on $x_N$ ensures $\varphi(\D)\subset \D^2$. On the other hand,
$\Re \varphi_1(\zeta)\le -\delta + |\zeta| \frac{\delta}{\psi^{-1}(\delta)}$,
so when $|\zeta|<\psi^{-1}(\delta)$, $\Re \varphi_1(\zeta)<0$, and
when $|\zeta|\ge \psi^{-1}(\delta)$,  $\Re \varphi_1(\zeta) \le -\delta
+ \psi(|\zeta|) =  -\delta
+ \psi(|\varphi_2(\zeta)|)$.

\subsection{Proof of Proposition \ref{onehalf}}

The lower estimate in Part (1) follows from Remark 1(b).
For the  upper estimate,
by inclusion, it is enough to prove it for $\psi(x)=c_0x^{1/2}$.
Proposition \ref{tangent} (1) takes care of the case $|x_T|>\max(1,c_0^{-1})|x_N|$.

Using the fact that
$\kappa_{G_\psi}(p_\delta;(x_N,x_T))=
|x_N|\kappa_{G_\psi}(p_\delta;(1,\frac{x_T}{x_N}))$, we may assume $x_N=1$.

Define
\[
\varphi(\zeta):=(-\delta+\zeta,
 x_T \zeta + \frac{x_T}{c_0^2 |x_T|}\zeta^2),
\]
with the convention that $\frac{x_T}{ |x_T|}=1$ when $x_T=0$.

Since we may assume $|x_T|\le \frac1{c_0}$, we see that $\varphi(\zeta)\in\D^2$ when
$|\zeta|\le \min(1-\delta_0, c_0\frac{-1+\sqrt5}2)$.
When $\Re \zeta<0$, $\Re \varphi_1(\zeta) < \psi(|\varphi_2(\zeta)|)$
trivially. When $\Re \zeta\ge 0$,
\[
\psi(|\varphi_2(\zeta)|) = \psi\left(\left|\frac{x_T}{ |x_T|}\zeta\right| \left| |x_T|
+ \frac{1}{c_0^2 }\zeta\right|\right)
\ge \psi(\frac{1}{c_0^2}|\zeta|^2) = |\zeta|
> \Re \varphi_1(\zeta).
\]
So by an easy change of variable,
$\kappa_{G_\psi}(p_\delta;(1,x_T)) \le \max((1-\delta_0)^{-1}, \frac{1+\sqrt5}{2c_0})=:C_1$.
\smallskip

To prove Part (2), since the case $x_N=0$ is covered by Remark 1(b), again we may assume $x_N=1$.

\begin{lem}
\label{lowlem}
Let $\varphi=(\varphi_1,\varphi_2)$ be a holomorphic map
from $\D$ to $G_\psi$ with $\varphi(0)=p_\delta$
and $\varphi'(0)= \lambda (1,x_T)$, with $\lambda >0$.

Let $0<r\le 1$, and $M:=\psi(r(\lambda |x_T | +r))$. Then
\begin{equation}
\label{lambdamaj}
\lambda \le \frac{2(M+\delta)}{r}.
\end{equation}
\end{lem}

\begin{proof}
Let $\varphi_2(\zeta)=\zeta h(\zeta)$, with $h(\D)\subset \D$ and $h(0)= \lambda x_T $.
This implies
\[
|h(\zeta)| \le \frac{|h(0)|+|\zeta|}{1+|h(0)||\zeta|} \le \lambda |x_T | +|\zeta|.
\]

It follows that $\max_{|\zeta|\le r} |\varphi_2(\zeta)| \le r(\lambda |x_T | +r)$,
and so $\Re \varphi_1(\zeta)<M$ for $|\zeta|<r$.
Then  elementary computations show that the function $f(\zeta):= \dfrac{\varphi_1(\zeta)+\delta}{\varphi_1(\zeta)-2M-\delta}$
verifies $f(0)=0$, $f(D(0,r))\subset \D$, and $|f'(0)|= \left| \dfrac{\varphi'_1(0)}{2M+2\delta}\right|$.
From the Schwarz Lemma, we deduce \eqref{lambdamaj}.
\end{proof}

Observe that if $\frac{\psi(x)}{\sqrt x}$ is an increasing function on $(0;1)$, then
$\frac{\sqrt{\psi^{-1}(\delta)}}\delta \le
\frac{t}{ \psi(t^2)} $ if and only if $t \le \sqrt{\psi^{-1}(\delta)}$.
So we may split the proof of Proposition \ref{onehalf} (2) into two cases.

{\it Case $|x_T|\le \sqrt{\psi^{-1}(\delta)} $.}

Let $r:= \frac12 \sqrt{\psi^{-1}(\delta)}$. Since $\varphi_1(\D)\subset \D$, $\lambda \le 1$ so
$r(\lambda |x_T | +r) \le \frac34 \psi^{-1}(\delta)$, so $M \le \delta$
and Lemma \ref{lowlem} implies $\lambda \le 8\delta/  \sqrt{\psi^{-1}(\delta)}$, q.e.d.

{\it Case $|x_T|\ge  \sqrt{\psi^{-1}(\delta)} $.}

Let $r:=  \frac12 |x_T|$, so $M \le \psi (|x_T|^2)$, so
it follows from Lemma \ref{lowlem} that $\lambda \le 4(\delta+\psi (|x_T|^2))/|x_T| \le
8\psi (|x_T|^2)/|x_T| $ by the hypothesis on
$|x_T|$.

\section{Proof of Theorem \ref{main}}
\label{mainpf}

\subsection{Lower estimates.}

We use the same method as in the proof of Proposition \ref{onehalf} (2).
We assume $x_N=1$ and let $\varphi$ be as in Lemma \ref{lowlem}, $r$ to be chosen, $\lambda>0$; we need
to bound $\lambda$ from above.
\smallskip

\noindent{\bf Part (1).} We begin by analysing the function $F_3$.

\begin{lem}
\label{maxzone}
If $\frac{\psi(x)}{\sqrt x}$ is an increasing function on $(0;1)$,
if furthermore $\psi_1$ is a decreasing function on $(0;1)$, then
$\frac{\sqrt{\psi^{-1}(\delta)}}{\delta} \le
\frac{8 t}{\sqrt{\psi_1^{-1}\left(\frac{1}{8t}\right)}}$ if and only if
$t\le\frac{\psi^{-1}(\delta)}{8\delta} $.
\end{lem}

\begin{proof}
First notice
that $\delta \mapsto \dfrac{\sqrt{\psi^{-1}(\delta)}}{\delta} =  \dfrac{\sqrt{\psi^{-1}(\delta)}}{\psi\left( \psi^{-1}(\delta)\right)}$
is a decreasing function, so the desired inequality will hold when $\delta \ge \Delta(t)$, where $\Delta(t)$ stands for
the value achieving equality between the two quantities.

Now if $t = \frac{\psi^{-1}(\delta)}{8\delta} =  \frac1{8\psi_1(\psi^{-1}(\delta))} $, then
\[
\psi_1^{-1}\left(\frac{1}{8t}\right) = \psi^{-1}(\delta), \mbox{ so } \frac{8 t}{\sqrt{\psi_1^{-1}\left(\frac{1}{8t}\right)}}
=  \frac{\psi^{-1}(\delta)}{\delta \sqrt{ \psi^{-1}(\delta)}} = \frac{\sqrt{ \psi^{-1}(\delta)}}{\delta }.
\]
So $\Delta^{-1}(\delta) = \frac{\psi^{-1}(\delta)}{8\delta}$, and since
$\delta \mapsto \dfrac{\psi^{-1}(\delta)}\delta =  \dfrac{\psi^{-1}(\delta)}{\psi\left( \psi^{-1}(\delta)\right)}$
is an increasing function under the hypothesis of Part (1) of Theorem \ref{main}, $t\mapsto\Delta(t)$
is increasing as well, and $\delta \ge \Delta(t)$ is equivalent to $\frac{\psi^{-1}(\delta)}{8\delta} \ge t$.
\end{proof}

Lemma \ref{maxzone} allows us to split the proof into two cases.

{\it Case $|x_T|\le\frac{\psi^{-1}(\delta)}{8\delta}$.}

Let $r:=\sqrt{\frac{\psi^{-1}(\delta)}{2}}$.

If $\lambda |x_T|  \le r$, then as in the proof of Lemma \ref{lowlem},
$r(\lambda |x_T | +r) <2r^2$, so $M< \delta$.
From Lemma \ref{lowlem} we get $\lambda \le 4 \delta/r = 4 \sqrt2 \delta/\sqrt{\psi^{-1}(\delta)}$.

If $\lambda |x_T| \ge r$,
then for $|\zeta|<r$,
$|\varphi_2(\zeta)|<2r \lambda |x_T|$, so
$M\le \psi(2r\lambda |x_T|)$, so by Lemma \ref{lowlem}
\[
 \lambda \le 2 \frac{\psi(2r\lambda |x_T|)}r
\le 2 \frac{\psi(2\frac{r}{8\delta}\lambda  \psi^{-1}(\delta))}r
\le
2 \frac{r}{4\delta}\lambda \frac{\psi( \psi^{-1}(\delta))}{r}=\frac\lambda2,
\]
if we assume $\frac{r}{4\delta}\lambda \ge 1$,
since $\frac{\psi(x)}{x}$ is decreasing; but this
is a contradiction, so $\lambda  \le \frac{4\delta}{r}$ once again.

{\it Case $\frac{\psi^{-1}(\delta)}{8\delta}\le |x_T| \le 1$.}

Let $r:= \sqrt{\psi_1^{-1}\left(\frac{1}{|x_T|}\right)}$.
If $\lambda |x_T|\le r$, then we are done. So assume
$\lambda |x_T|\ge r$, then for $|\zeta|\le r$, then $M \le \psi(2r\lambda |x_T|)$,
and by Lemma \ref{lowlem}, $\lambda r\le 2(\delta+ \psi(2r\lambda |x_T|))$.

Since $ |x_T|\ge\frac{\psi^{-1}(\delta)}{8\delta}
=\frac1{8\psi_1(\psi^{-1}(\delta))}$, by hypothesis (H)
we have
\begin{multline*}
\delta \le \psi\left( \psi_1^{-1}\left(\frac{1}{8|x_T|}\right)\right)
\le \psi\left( K^3 \psi_1^{-1}\left(\frac{1}{|x_T|}\right)\right)
\\
\le
\max(\frac{K^3}2,1) \psi\left( 2 \psi_1^{-1}\left(\frac{1}{|x_T|}\right)\right)
\le
\max(\frac{K^3}2,1) \psi(2r^2)
\le \max(\frac{K^3}2,1) \psi(2r\lambda |x_T|).
\end{multline*}
So we have
$\lambda r \le C_3 \psi(2r\lambda |x_T|)$,
so $1 \le 2 C_3 |x_T| \psi_1(2r\lambda |x_T|)$,
equivalently, choosing $m\in \N$ such that $2^m\ge 2C_3$,
\[
\psi_1(2r\lambda |x_T|) \ge
\frac{1}{2 C_3}
\psi_1 \left( \psi_1^{-1}\left(\frac{1}{|x_T|}\right)\right)
\ge \frac{2^m}{2 C_3}
\psi_1 \left( K^{m}\psi_1^{-1}\left(\frac{1}{|x_T|}\right)\right)
\ge
\psi_1 \left( K^{m}\psi_1^{-1}\left(\frac{1}{|x_T|}\right)\right),
\]
so $2r\lambda |x_T| \le K^m r^2$, thus $\lambda \le \frac{K^m }{2|x_T|}r$, q.e.d.
\smallskip

\noindent{\bf Part (3).} If $|\lambda|\le 4\delta$, the conclusion is already obtained
since $\psi^{-1}(\delta)\le 1$.
Hence
we may choose $r=4\delta/|\lambda|$ in Lemma \ref{lowlem}, and get
that if there is a map $\varphi$ as in that Lemma,
\[
2\delta = \frac{r|\lambda|}2 \le \delta + \psi (r^2+|\lambda x_T| r),
\]
thus $\delta \le \psi (r^2+|\lambda x_T| r)$, so, since $\psi$ is increasing,
\[
\psi^{-1}(\delta) \le r^2+|\lambda x_T| r \le 2\max(r^2, 4\delta |x_T| ).
\]
However the hypothesis implies that $8\delta |x_T|< \psi^{-1}(\delta)$, so we must have $\psi^{-1}(\delta) \le 2 r^2$,
which means $\ds\frac1{|\lambda|} \ge \frac1{4\sqrt 2} \frac{\sqrt{\psi^{-1}(\delta)}}{\delta}$.
\smallskip

\noindent{\bf Part (2).} 

Proposition \ref{tangent} already covers the case where $|x_T|\ge |x_N|$. 
So we will assume $|x_T| < |x_N|$, and 
by homogeneity, $x_N=1$. 

Take $\varphi$ as in Lemma \ref{lowlem}. We must have $\lambda \le 1$ by the Schwarz Lemma applied to $\varphi_1$. For $\zeta\in \D$,
$\varphi_1(\zeta)=-\delta + \lambda \zeta + \zeta^2 \tilde\varphi_1(\zeta)$, $\tilde\varphi_1\in\mathcal O(\D)$ and
\[
|\tilde\varphi_1(\zeta)| \le \sup_{|\zeta| =1} |\varphi_1(\zeta)+\delta - \lambda \zeta| \le 1+\delta+\lambda \le 3.
\]
Thus $\Re \varphi_1(\zeta) \ge -\delta + \lambda \Re\zeta - 3 |\zeta|^2$.

On the other hand, by the proof of Lemma \ref{lowlem}, $|\varphi_2(\zeta)|\le \lambda |x_T ||\zeta| +|\zeta|^2$. Applying 
the condition $\Re \varphi_1(\zeta) \le |\varphi_2(\zeta)|$
to the value $\zeta=\sqrt \delta$, we obtain 
\[
\lambda \sqrt \delta - 4 \delta < \lambda |x_T |\sqrt \delta + \delta, \mbox{ i.e. }
\lambda < \frac{5 \sqrt \delta}{1-|x_T |},
\]
so we obtain the conclusion with $c= \frac15$. 

\subsection{Upper estimates}

For those estimates, we need to construct maps $\varphi: \D \longrightarrow G_\psi$
with $\varphi(0) =(-\delta,0)$, $\varphi'(0) = \lambda (1, x_T)$ and $\lambda$ as large
as we can (which is still a quantity tending to $0$
as $\delta$ tends to $0$).
\smallskip

\noindent{\bf Part (1).} By Lemma \ref{maxzone}, we may split the proof into two (slightly overlapping) cases.

{\it Case $|x_T| \le   \frac{\psi^{-1}(\delta)}{ \delta}  $.}

Let $\lambda:= \frac{ \delta}{2\sqrt{\psi^{-1}(\delta)}}$, 
and consider the map given, for $\zeta\in \D$, by
\[
\varphi(\zeta)= \left(-\delta+ \lambda  \zeta, \lambda x_T \zeta +\frac{\zeta^2}2 \right).
\]
It is enough to show that $\varphi(\D)\subset G_\psi$.

First the hypotheses imply that $\lambda |x_T| \le \frac12 \sqrt{\psi^{-1}(\delta)}$, so
\[
|\varphi_2(\zeta)| \le  \frac12 \sqrt{\psi^{-1}(\delta)} |\zeta|+\frac{|\zeta|^2}2 .
\]

When $|\zeta|< 2 \sqrt{\psi^{-1}(\delta)}$, then $\Re ( \lambda \zeta ) \le  \lambda|\zeta|  < \delta$,
so $\Re \varphi_1(\zeta) < 0 $ and $\varphi(\zeta)\in G_\psi$.

Suppose now  $2 \sqrt{\psi^{-1}(\delta)} \le |\zeta| <1  $. Then
$\lambda |x_T| \le \frac14 |\zeta| $, so
\[
\left| \lambda x_T \zeta +\frac{\zeta^2}2 \right| \ge \frac{|\zeta|^2}2 - \frac14 |\zeta|^2= \frac{|\zeta|^2}4.
\]
So in this case it will be enough to prove
\[
 \frac{\delta}{2\sqrt{\psi^{-1}(\delta)}} |\zeta| < \psi\left( \frac{|\zeta|^2}4\right) +\delta,
\]
for which it is enough to have
\[
\frac{\psi\left( (\frac{|\zeta|}2)^2\right) } {\frac{|\zeta|}2}\ge  \frac{\delta}{\sqrt{\psi^{-1}(\delta)}}
 =  \frac{\psi((\sqrt{\psi^{-1}(\delta)})^2)}{\sqrt{\psi^{-1}(\delta)}},
\]
which is satisfied since  $|\zeta|/2 \ge  \sqrt{\psi^{-1}(\delta)}$,
and $\frac{\psi(x^2)}x= \frac{\psi(x^2)}{\sqrt{x^2}}$ is increasing.

Notice that in this case, we have not used the hypothesis that $\psi_1$ be decreasing,
so that this also proves the required upper estimate for Part (3).

{\it Case $|x_T|\ge\frac{\psi^{-1}(\delta)}{8\delta}  $.}

Assume $\delta\le \frac12$.
Rescaling the disc with $\zeta=(1-\delta)\zeta'$,
at the cost of an innocuous factor $ (1-\delta)^{-1}$, it will be enough to prove that $\varphi(\zeta)\in G_\psi$
for $|\zeta|<1-\delta$ and
\[
\varphi(\zeta):=(-\delta+\lambda\zeta,
\lambda x_T \zeta + \frac{x_T}{|x_T|}\zeta^2),
\mbox{ with }
\lambda:= \frac{1}{|x_T|}\sqrt{\psi_1^{-1}\left(\frac{1}{|x_T|}\right)}.
\]
Note that $\lambda \le \psi(1)$ so that $\varphi(\zeta) \in \D^2$ for  $|\zeta|\le \frac12 \min(1,\frac1{\psi(1)})$.

Indeed, for $y\ge 1$,
\[
\psi(1)=\frac{\psi(1^2)}1 \ge \frac{\psi(y^{-2})}{y^{-1}} = \frac{\psi_1(y^{-2})}{y},
\]
so, since $\psi_1$ is decreasing, $\psi_1^{-1}(\psi(1)y) \le y^{-2}$. Now apply this to $y= \frac1{\psi(1)|x_T|}$.

For $\Re \zeta \le 0$, clearly $\varphi(\zeta)\in G_\psi$.
For $\Re \zeta > 0$,
\[
|\varphi_2(\zeta)|\ge \left| \frac{x_T}{|x_T|}\zeta \right|
 \left| \lambda |x_T| + \zeta \right|
 \ge |\zeta| \max( |\zeta|, \lambda |x_T|).
\]
On the other hand, $\Re \varphi_1(\zeta) \le \lambda |\zeta|$.
So we are reduced to proving $\psi(|\zeta| \max( |\zeta|, \lambda |x_T|))> \lambda |\zeta|$.

{\it Subcase 1. $|\zeta|\le \lambda |x_T|$.}

It is enough to verify $\lambda |\zeta|\le \psi(\lambda |x_T||\zeta|)$,
i.e. $\frac{1}{|x_T|}\le \psi_1(\lambda |x_T||\zeta|)$.
But $\psi_1$ is a decreasing function, so it is enough to
check this for $ |\zeta|= \lambda |x_T|$, where we obtain
\[
\psi_1(\lambda^2 |x_T|^2) = \psi_1(\psi_1^{-1}\left(\frac{1}{|x_T|}\right))= \frac{1}{|x_T|}.
\]

{\it Subcase 2. $|\zeta|\ge \lambda |x_T|$.}

It is enough to verify $\lambda |\zeta|\le \psi(|\zeta|^2)$,
but since $\frac{\psi(x)}{\sqrt x}$ is an increasing function,
\[
\frac{ \psi(|\zeta|^2)}{|\zeta|}
\ge \frac{ \psi(\lambda^2 |x_T|^2)}{\lambda |x_T|}
= \lambda |x_T| \psi_1(\lambda^2 |x_T|^2)= \lambda.
\]

\noindent{\bf Part (2).} 

This proof will gather piecemeal results on a number of different cases.  We would have liked to find a more elegant argument.


Proposition \ref{tangent} already covers the case where $|x_T|\ge |x_N|$. 
On the other hand, Part (3) in particular proves the  estimate for the case $|x_T|\le \frac18 |x_N|$. 
So we will assume  $\frac18 |x_N| < |x_T| < |x_N|$, and by homogeneity, $x_N=1$. Also assume
$\delta \le \frac12$.

{\it Case 1.} $0<\delta<(1-|x_T|)^2$.

Define
\[
\alpha:=  \frac{(1-|x_T|)^{2}}{2\delta}, \quad
\varphi(\zeta):=  \left(-\delta + \zeta,
{x_T} \zeta  + \alpha \frac{x_T}{|x_T|}\zeta^2 \right).
\]
We claim that $\varphi(\zeta)\in G_\psi$ for
$|\zeta|< \frac12 \sqrt\d (1-|x_T|)^{-1} $,
which will prove the estimate once we renormalize to the unit disc.

First note that $|\varphi_1(\zeta)| \le \delta + \frac12  <1$ since $\delta
\le \frac12$, and
$|\varphi_2(\zeta)| \le \frac{|x_T|\sqrt\d}{2(1-|x_T|)}    + \frac1{8}
\le \frac12 + \frac1{8}  <1$.

We still must check the condition $\Re \varphi_1( \zeta) < |\varphi_2(\zeta)|$.
 For $ \Re \zeta < \delta$, this is immediate.
For $x:= \Re \zeta \ge \delta >0$, note that
\[
|\varphi_2(\zeta)|= |\zeta|\left| |x_T| + \alpha \zeta \right| \ge
x(|x_T| +\alpha x),
\]
so that the required inequality will be satisfied if
\[
x(|x_T| +\alpha x) -(x-\delta) = \alpha x^2 -(1-|x_T|)x + \delta
=  \frac{1}{2\delta}
\left( ((1-|x_T|)x)^{2} - 2\delta  (1-|x_T|)x + 2\delta^2 \right) >0,
\]
which clearly holds.
\smallskip

{\it Case 2.} $\delta \ge (1-|x_T|)^2$.

First notice that since $\frac18 \le |x_T| \le 1$ and $\frac{1-|x_T|}{\sqrt\delta} \le 1$,
$\frac18 \le \max\left( \frac{1-|x_T|}{\sqrt \delta}, |x_T| \right) \le 1$, so it will
be enough to find a map $\varphi\in\mathcal O(\D,G_\psi)$ with 
$\varphi(0)=p_\delta$
and $\varphi'(0)= \lambda (1,x_T)$, with $\lambda $ some nonzero constant. 
Let 
\[
\varphi(\zeta):=  \left(-\delta + \lambda \zeta,
\lambda {x_T} \zeta  + \frac14 \frac{x_T}{|x_T|}\zeta^2 \right).
\]
We claim that $\varphi(\D)\subset G_\psi$ for $\lambda = \min (1-\delta_0, \frac18)$. 
Such a choice of $\lambda$ in particular implies that $\varphi(\D)\subset \D^2$.

Let $\zeta=x+iy$, $x,y\in\R$. 
Note first that $\Re \varphi_1(\zeta) <0$ when $x<\delta/\lambda$, so that it is enough to look at 
values of $\zeta$ with $\lambda x \ge \delta$. We then need to check $(\Re \varphi_1(\zeta))^2 < |\varphi_2(\zeta)|^2$,
i.e.
\[
(\lambda x - \delta)^2 < (x^2+y^2) \left( (\lambda |x_T| + \frac{x}4)^2 + \frac{y^2}{16} \right).
\]
This will hold for all $\zeta\in \D$ if and only if it holds for all $x\in(\frac{\delta}{\lambda};1)$, $y=0$, so we need to check
\[
\lambda^2 x^2 (1-|x_T|^2) + \delta^2 < 2 \lambda x  \delta + \lambda |x_T| \frac{x^3}2 + \frac{x^4}{16}.
\]
The left hand side is dominated by $2 \lambda^2 x^2 \sqrt \delta + \lambda x \delta $, so it is enough to check
\[
2 \lambda^2 x^2 \sqrt \delta < \lambda x  \delta + \lambda |x_T| \frac{x^3}2 + \frac{x^4}{16}.
\]

{\it Case 2.1.} $x< 4 \sqrt \delta$.

It will be enough to verify $8 \lambda^2 x \delta \le \lambda x  \delta$, which follows from $\lambda \le \frac18$.

{\it Case 2.2.} $x\ge 4 \sqrt \delta$.

Then 
\[
\lambda |x_T| \frac{x^3}2 \ge  \lambda \frac{x^3}{16}\ge \frac1{4} \lambda x^2 \sqrt \delta \ge 2 \lambda^2 x^2 \sqrt \delta,
\]
whenever $\lambda \le \frac18$. Since $\frac{x^4}{16}>0$, this finishes the proof.
\smallskip

\noindent{\bf Part (3).} The case $|x_T| \le   \frac{\psi^{-1}(\delta)}{ \delta}$
has been covered by the proof of Part (1), so we may assume
$|x_T| \ge \frac{\psi^{-1}(\delta)}{ \delta}  $.
Assume $\delta_0\le \frac12$.

Consider the map given, for $\zeta\in \D$, by
\[
\varphi(\zeta)= \left(-\delta+ \frac{1}{x_T} \zeta, \zeta\right).
\]
We want to see that $\varphi(\zeta)\in G_\psi$ when $|\zeta| <  c \frac{\delta}{\sqrt{\psi^{-1}(\delta)}}$.
This condition implies that $|\zeta|/|x_T| \le  \frac{ \delta |\zeta|}{\psi^{-1}(\delta)} <\frac12$
for $c< \delta_0^{-1/2} \psi^{-1}(\delta_0)^{3/2}$.
It is thus enough to see that
\begin{equation}
\label{suffx}
-\delta+ \Re\left( \frac{1}{x_T} \zeta \right) \le  -\delta+   \frac{ \delta}{\psi^{-1}(\delta)} |\zeta| < \psi(|\zeta|).
\end{equation}
This is obviously satisfied for $|\zeta|\le \psi^{-1}(\delta)$; for
$|\zeta|> \psi^{-1}(\delta)$,  since $\frac{\psi(x)}{x}$ is increasing,
\[
\frac{\delta}{\psi^{-1}(\delta)} \le \frac{\psi(|\zeta|)}{|\zeta|} < \frac{\psi(|\zeta|)}{|\zeta|} +
\frac{\delta}{|\zeta|},
\]
which proves the property.

{}

\end{document}